\documentclass[english]{amsart}

\usepackage{esint}
\usepackage[svgnames]{xcolor} 
\usepackage[colorlinks,citecolor=red,pagebackref,hypertexnames=false,breaklinks]{hyperref}
\usepackage{pgf,tikz}
\usepackage{pdfsync}

\usepackage{dsfont}
\usepackage{url}
\usepackage[utf8]{inputenc}
\usepackage[T1]{fontenc}
\usepackage{lmodern}
\usepackage{babel}
\usepackage{mathtools}  
\usepackage{amssymb}
\usepackage{lipsum}
\usepackage{mathrsfs}
\usepackage{color}
\usepackage[skip=4pt plus1pt, indent=12pt]{parskip}

\newtheorem{theorem}{Theorem}[section]
\newtheorem{proposition}{Proposition}[section]
\newtheorem{lemma}{Lemma}[section]

\newtheorem{corollary}{Corollary}[section]

\numberwithin{equation}{section}

\title[Fractional anisotropic Calder\'on problem]{Fractional anisotropic Calder\'on problem on complete Riemannian manifolds}

\author[Mourad Choulli]{Mourad Choulli}
\address{Universit\'e de Lorraine, France}
\email{mourad.choulli@univ-lorraine.fr}

\author[El Maati Ouhabaz]{El Maati Ouhabaz}
\address{University of Bordeaux, Institut de Math\'ematiques de Bordeaux, France}
\email{Elmaati.Ouhabaz@math.u-bordeaux.fr}

\date{}

\begin{document}

\begin{abstract}
We prove that the metric tensor $g$ of a complete Riemannian manifold is uniquely determined, up to isometry, from the knowledge of a local source-to-solution operator associated with a fractional power of  the Laplace-Beltrami operator $\Delta_g$. Our result holds under the condition that the metric tensor $g$ is known in an arbitrary small subdomain. We also consider the case of closed manifolds and provide an improvement of the main result in \cite{FGKU}.
 \end{abstract}

\subjclass[2010]{35R30, 35R11, 35R01}

\keywords{Fractional Laplace-Beltrami operator, Fractional anisotropic Calder\'on problem, local source-to-solution operator.}

\maketitle

%\tableofcontents

\section{Introduction}

Let   $(M,g)$ be a complete Riemannian manifold  without boundary, it is connected and has dimension $n\ge 2$. 
We denote by $\Delta_g$ the  associated  Laplace-Beltrami operator and by  $d\mu_g=\sqrt{\mathrm{det}(g)}dx_1\ldots dx_n$  the Riemannian measure.
We use the convention that $\Delta_g$ is nonnegative. It is the selfadjoint operator associated with the quadratic form
\[
 {\mathfrak a}(u,u) = \int_M | \nabla u |^2 \, d\mu_g, \quad u \in H^1(M).
\]
We use the classical notation $D(\Delta_g)$ for its domain and $R(\Delta_g)$ for its range.  Let $e^{-t\Delta_g}$ denote the corresponding semigroup on $L^2(M) = L^2(M, d\mu_g)$.  For a given $\alpha \in (0, 1)$, we recall 
\[
\Delta_g^{-\alpha}f =\frac{1}{\Gamma(\alpha)}\int_0^\infty t^{\alpha-1}e^{-t\Delta_g}fdt,\quad f\in R(\Delta_g).
\]
One then defines 
\[
\Delta_g^\alpha f = \Delta_g \Delta_g^{-(1-\alpha)}f = \frac{1}{\Gamma(1-\alpha)}\int_0^\infty t^{-\alpha} e^{-t\Delta_g} \Delta_g fdt,\quad f \in D(\Delta_g).
\]

Let  $O$ be an open nonempty subset of $M$ and define
\[
\tilde{\mathcal D}_0^g (O) = R\left({\Delta_g}_{\vert C_0^\infty (O)}\right) = \{ \Delta_g f, \ f \in C_0^\infty (O) \}.
\]
 Our main purpose is to establish the following result.
 
 \begin{theorem}\label{theorem0}
Let $(M_j,g_j)$, $j=1,2$, be two complete Riemannian manifolds and $0<\alpha<1$. Assume that there exists a nonempty open subset $O_j$ of $M_j$, $j=1,2$, such that $(O_1,g_1)=(O_2,g_2):= (O,g)$ and
\begin{equation}\label{th1}
\chi_O\Delta_{g_1}^{-\alpha}f=\chi_O\Delta_{g_2}^{-\alpha}f,\quad f\in \tilde{\mathcal D}_0^{g_1}(O_1)=\tilde{\mathcal D}_0^{g_2}(O_2):=\tilde{\mathcal D}_0.
\end{equation}
Then there exists a diffeomorphism $F: M_1\rightarrow M_2$ such that $F^\ast g_2=g_1$.
\end{theorem}
In Theorem \ref{theorem0},  $\chi_O$ denotes the characteristic function of $O$.

Let us see that the result of Theorem \ref{theorem0} can be interpreted in terms of a source-to-solution operator. Assume first that $M_j$ is not compact. In that case, the Laplacian is injective.  Approximating $\Delta_{g_j}$ by $\epsilon +\Delta_{g_j}$, $\epsilon >0$, we easily see that for each $f\in \tilde{\mathcal D}_0^{g_j}(O_j)$ there exists  $u_j=u_j(f)\in D(\Delta_{g_j})$ satisfying $\Delta_{g_j}^\alpha u_j=f$. This $u_j$ is unique since $\Delta_{g_j}$ is injective.  Furthermore, $u_j=\Delta_{g_j}^{-\alpha}f$. As a consequence the source-to-solution operator 
\[
\mathcal{S}_j:f\in \tilde{\mathcal D}_0^{g_j}(O_j)\mapsto u_j(f)_{|O_j}\in D(\Delta_{g_j})
\]
is well defined. The assumption \eqref{th1} is then equivalent to 
$\mathcal{S}_1=\mathcal{S}_2$.

 If $M_j$ is compact then we consider $\Delta_{g_j}$ on the orthogonal space to the constant function $1$. We then proceed as above. Let us also note for clarity  that in this case, $1 \notin R(\Delta_{g_j})$. Indeed if $u_j \in D(\Delta_{g_j})$ is such that 
 $\Delta_{g_j} u_j = 1$,  then for all $t > 0$, $ \Delta_{g_j}e^{-t\Delta_{g_j}} u_j = 1$. Taking the $L^2$ norm and using the analyticity of the semigroup, it follows that $\| 1 \|_2 \le Ct^{-1} \| u_j \|_2$ for all $ t >0$, where $C>0$ is a constant.  One has a contradiction by letting $t \to  \infty$.

Theorem \ref{theorem0} was recently  established in \cite{FGKU} in the case of  closed Riemannian manifolds, that is compact manifolds without boundaries. One of the key tools in the proof of the  main result in \cite{FGKU} is a pointwise Gaussian upper bound for the heat kernel of the Laplace-Beltrami operator. The  Gaussian bounds are however  not  valid for an arbitrary complete non compact Riemannian  manifold. Instead, we use the so-called Davies-Gaffney estimate (an $L^2$-off diagonal bound) which, in contrast to the Gaussian upper bound, is valid for any complete Riemannian manifold.
We mention that the result of \cite{FGKU} generalizes an earlier one in \cite{Fe} in which the author assumes, in addition of the assumptions in \cite{FGKU}, that $(O_j,g_j{_{|O_j}})$ belongs locally to a Gevrey class, $j=1,2$. 

Obviously, $\tilde{\mathcal D}_0 \subset C_0^\infty (O)$. Note that we have formulated the equality \eqref{th1} for $f\in \tilde{\mathcal D}_0$.
In the setting of compact manifolds as in Theorem \ref{theorem0.0} below, the Laplacians are invertible operators on the orthogonal spaces to the constants. In particular, $\Delta_{g_1}^{-\alpha}$ and $\Delta_{g_2}^{-\alpha}$ are bounded operators on those spaces and we can then write the equality \eqref{th1} for  all $f \in C_0^\infty(O)$ which is orthogonal to $1$.  

Beside Theorem \ref{theorem0} we establish another  improvement of the main result in \cite{FGKU}. Prior to stating this improvement we need to introduce some notations.

For $j=1,2$, let $(M_j,g_j)$ be a closed manifold and denote by $(\lambda_k^j,\phi_k^j)_{k \ge 0}$ the sequence of eigenvalues and the corresponding eigenfunctions so that $(\phi_k^j)_{k \ge 0}$ is an orthonormal basis of $L^2(M_j)$. Suppose now  that  there exists a nonempty open subset $O_j$ of $M_j$, $j=1,2$, such that $(O_1,g_1)=(O_2,g_2):=(O,g)$. We define the set
\[
N_\ell =\left\{f\in C_0^\infty (O), \  \int_O \phi_k^j\overline{f}d\mu_g=0,\; 0\le k\le \ell,\; j=1,2\right\},\quad \ell \ge 0.
\]

 \begin{theorem}\label{theorem0.0}
Let $(M_j,g_j)$, $j=1,2$, be two closed Riemannian manifolds, $0<\alpha<1$ and $\ell \ge 0$. Assume that there exists a nonempty open subset $O_j$ of $M_j$, $j=1,2$, such that $(O_1,g_1)=(O_2,g_2):=(O,g)$ and
\begin{equation}\label{th1-1}
\chi_O\Delta_{g_1}^{-\alpha}f=\chi_O \Delta_{g_2}^{-\alpha}f,\quad f\in N_\ell .
\end{equation}
Then there exists a diffeomorphism $F:M_1\rightarrow M_2$ such that $F^\ast g_2=g_1$.
\end{theorem}

It is well known that two closed isometric manifolds are isospectral, that is they have the same spectrum. Therefore,  it follows from  Theorem \ref{theorem0.0} that the knowledge of $\Delta_{g_j}^{-\alpha}$ on a possibly small common open subset  $O$ and outside a finite number of eigenspaces allows to recover equality of all eigenvalues. This fact was not contained  in the main result of \cite{FGKU} which considers the case $\ell = 0$, only (see also \cite[Proposition 2.1]{Fe} where isospectrality is proved under a stronger condition than in \cite{FGKU}). 

Other types of inverse problems for fractional operators have  attracted attention in recent years. 
 Precisely, these inverse problems consist in the determination of the potential or the isotropic conductivity  in fractional Schr\"odinger operators from interior measurements. Without being exhaustive we mention the following references \cite{BGU,Chien, CLR,CGR,GRSU,GSU,QU,Ru,RSa1,RSa2,RSi}.
In order to give an idea of what kind of inverse problems were treated in some of these  references, we discuss briefly the problem of determining the potential in a fractional Sch\"odinger equation from an analogue of the Dirichlet-to-Neumann map. Denote by $(-\Delta)^\alpha=\mathrm{Op}(|\xi|^{2\alpha})$ the fractional Laplacian of $\mathbb{R}^n$ of order $\alpha$. Pick $\Omega$ a bounded domain of $\mathbb{R}^n$. Let $q\in L^\infty(\Omega)$ and consider the  fractional Sch\"odinger equation
\begin{equation}\label{S1}
(-\Delta)^\alpha u+qu=0\; \mathrm{in}\; \Omega, \quad u_{|\Omega_e}=f,
\end{equation}
where $\Omega_e=\mathbb{R}^n \setminus \overline{\Omega}$.

It is shown in \cite{GSU} that the following analogue of Dirichlet-to-Neumann map
\[
\Sigma_q:f\in H^\alpha (\Omega_e)\mapsto H^\alpha(\Omega_e)^\ast: f\mapsto (-\Delta)^\alpha u_{|\Omega_e}
\]
is well defined provided that $0$ is not an eigenvalue for exterior Dirichlet problem for $(-\Delta)^\alpha +q$. 

Let $W_j$ be an open subset of $\Omega_e$, $j=1,2$. The main result in \cite{GSU} shows that $f\in C_0^\infty (W_1)\mapsto \Sigma_q(f)_{|W_2}$ determines uniquely $q$. 

In contrast to the classical Calder\'on problem, based on geometric optic solutions, the uniqueness theorem in \cite{GSU} relies essentially on the unique continuation property of the non local operator $(-\Delta)^\alpha$.

We also mention the references \cite{CRZ,RZ1,RZ2} where the recovery of variable kernels of elliptic fractional order operators in noncompact Euclidean setting is discussed.

We finally point out that inverse spectral problems associated with fractional elliptic operators are challenging problems.
A  uniqueness result for the problem consisting in determining the potential in a fractional Schr\"odinger operator from an interior spectral data was established in \cite{Ch}.

\section{Proof of the main result}

\subsection{Relationship to the heat equation: an abstract setting}

Let $\mathcal{E}_j$, $j=1,2$, be a (complex) Banach space with norm $\|\cdot \|_j$ and $-A_j$  the infinitesimal generator of a bounded analytic semigroup $e^{-tA_j}$ (see Subsection \ref{sub1} of Appendix \ref{appendixA}  for the definition). Set 
\begin{equation}\label{e1}
\|e^{-tA_j}\|_{\mathscr{B}(\mathcal{E}_j)}\le \Lambda_j ,\quad t\ge 0,
\end{equation}
for some constant  $\Lambda_j \ge 1$. 

Pick $0<\alpha <1$.  As usual, define 
\begin{equation}\label{e0}
A_j^{-\alpha}f =\frac{1}{\Gamma(\alpha)}\int_0^\infty t^{\alpha-1}e^{-tA_j}fdt,\quad f\in R(A_j).
\end{equation}

Assume that there exist $D_0$ and $D$, two nontrivial subspaces of $\mathcal{E}_1\cap \mathcal{E}_2$, so that 
\[
\|f\|_1=\|f\|_2:=\|f\|,\quad f\in D_0\cup D. 
\]
Further, suppose that the following assumptions hold

\noindent $\bf{(h1)}$ $\displaystyle D_0\subset \bigcap_{k\ge 1}D(A_j^k)$, $j=1,2$.
\\
$\bf{(h2)}$ $D_0$ is invariant  under $A_j^k$ for every $k\ge 1$,  that is,  $A_j^kD_0\subset D_0$, $j=1,2$.

In the case where $A_1{_{|D_0}}=A_2{_{|D_0}}$ we set 
\[
\tilde{\mathcal D}_0 =R(A_1{_{|D_0}})=R(A_2{_{|D_0}})  \subset D_0.
\]

Let $B_j\in \mathscr{B}(\mathcal{E}_j,D)$ and introduce the following definition that we call weak Davies-Gaffney estimate.

\noindent $\bf (WDG)$  For each $f\in D_0$ there exists two constants $C_j=C_j(A_j,f,B_j)>0$ and $\mu_j=\mu_j (A_j,f,B_j)>0$ such that
\[
\|B_je^{-tA_j}f\|\le C_je^{-\mu_j/t },\quad \; 0<t\le 1.
\] 

The following lemma consists in a generalization, in an abstract setting, of results already appearing in \cite[Section 2]{Fe} and \cite[Subsection 2.1]{FGKU}.

\begin{lemma}\label{lemma1}
Assume that
\begin{equation}\label{e2.0}
A_1f=A_2f \quad {\rm and}  \quad B_1 f = B_2 f, \quad f\in D_0.
\end{equation}
Then for all $f\in D_0$, $k\ge 1$ and $\ell \ge 0$, we have
\begin{equation}\label{e3}
t^{-k+\alpha-1}\frac{d^\ell}{dt^\ell}[B_1e^{-tA_1}-B_2e^{-tA_2}]f\in L^1((0,\infty),D) 
\end{equation}
and 
\begin{align}
&B_1A_1^{-\alpha}A_1^kf-B_2A_2^{-\alpha}A_1^kf \label{e4}
\\
&\hskip 1.5cm =\frac{1}{\Gamma(\alpha)}\prod_{j=0}^{k-1}(\alpha-1-j)\int_0^\infty t^{-k+\alpha-1}[B_1e^{-tA_1}-B_2e^{-tA_2}]f.\nonumber
\end{align}
\end{lemma}

\begin{proof}
First, note that, in light of $\bf{(h2)}$, \eqref{e2.0} implies
\begin{equation}\label{e2}
A_1^kf=A_2^kf,\quad k\ge 1,\; f\in D_0.
\end{equation}
Pick $k\ge 1$, $\ell \ge 0$ and $f\in D_0$. Let $\Lambda_j$  be the constant appearing in \eqref{e1}. If $\Lambda=\Lambda_1\|B_1\|+\Lambda_2\|B_2\|$, then we have
\[
\left\|t^{-k+\alpha-1}\frac{d^\ell}{dt^\ell}[B_1e^{-tA_1}-B_2e^{-tA_2}]f\right\|\le \Lambda t^{-k+\alpha-1}\max_{j=1,2}\|A_j^\ell f\|= \Lambda t^{-k+\alpha-1}\|A_1^\ell f\|.
\]
Thus
\begin{equation}\label{e5}
t^{-k+\alpha-1}\frac{d^\ell}{dt^\ell}[B_1e^{-tA_1}-B_2e^{-tA_2}]f\in L^1((1,\infty),D).
\end{equation}
Next, Taylor's formula yields
\[
\frac{d^\ell}{dt^\ell}e^{-tA_j}f=\sum_{m=0}^{k-1} \frac{t^m}{m!}\frac{d^{\ell+m}}{dt^{\ell+m}}e^{-tA_j}f{_{|t=0}}+\frac{t^k}{k!}\int_0^1(1-s)^{k-1}\frac{d^{k+\ell}}{dt^{k+\ell}}e^{-stA_j}f, \quad j=1,2.
\]
From \eqref{e2} we obtain
\[
\frac{d^{\ell+j}}{dt^{\ell+m}}e^{-tA_1}f{_{|t=0}}=-A_1^{\ell+m}f=-A_2^{\ell+m}f=\frac{d^{\ell+m}}{dt^{\ell+m}}e^{-tA_2}f{_{|t=0}}.
\]
Therefore
\begin{align*}
\frac{d^\ell}{dt^\ell}[B_1e^{-tA_1}&-B_2e^{-tA_2}]f
\\
&=-\frac{t^k}{k!}\int_0^1(1-s)^{k-1}[B_1e^{-stA_1}A_1^{k+\ell}f-B_2e^{-stA_2}A_2^{k+\ell}f]ds. 
\end{align*}
Hence 
\[
\left\|\frac{d^\ell}{dt^\ell}[B_1e^{-tA_1}-B_2e^{-tA_2}]f\right\|\le \frac{\Lambda t^k}{k!}\|A_1^{\ell+k} f\|,
\]
from which we derive that
\begin{equation}\label{e6}
t^{-k+\alpha-1}\frac{d^\ell}{dt^\ell}[B_1e^{-tA_1}-B_2e^{-tA_2}]f\in L^1((0,1),D).
\end{equation}
Combining \eqref{e0} and \eqref{e2}, we obtain
\begin{align*}
B_1A_1^{-\alpha}A_1^kf-B_2A_2^{-\alpha}A_1^kf&= \frac{1}{\Gamma(\alpha)}\int_0^\infty t^{\alpha-1}[B_1e^{-tA_1}-B_2e^{-tA_2}]A_1^kfdt
\\
&= \frac{1}{\Gamma(\alpha)}\int_0^\infty t^{\alpha-1}[B_1e^{-tA_1}A_1^kf-B_2e^{-tA_2}A_2^kf]dt
\\
&=\frac{(-1)^k}{\Gamma(\alpha)} \int_0^\infty t^{\alpha-1}\frac{d^k}{dt^k}[B_1e^{-tA_1}-B_2e^{-tA_2}]fdt. 
\end{align*}
We make  integration by parts $k$-times in the latest term. Note that by the boundedness of the semigroups and the fact that 
$\alpha \in (0, 1)$, the integrand at each  step goes to $0$ at infinity. Similarly, the integrand at each step is also $0$ at $0$. The reason is that by \eqref{e2.0} and $\bf{(h1)}$ and $\bf{(h2)}$,  the derivative at any order  of 
the function $t \mapsto \varphi(t) := [B_1e^{-tA_1}-B_2e^{-tA_2}]f$ is $0$ at $0$ and hence $\varphi(t)$ is bounded (up to a constant) by any power $t^j$ near $0$. Thus, we obtain  \eqref{e4}. 
\end{proof}

\begin{proposition}\label{proposition1}
Suppose that the assumptions of Lemma \ref{lemma1} and  (WDG) are satisfied. Let $f\in \tilde{\mathcal D}_0$.  Then the assumption 
\begin{equation}\label{e7}
B_1A_1^{-\alpha}f=B_2A_2^{-\alpha}f
\end{equation}
implies
\begin{equation}\label{e11}
B_1e^{-tA_1}f=B_2e^{-tA_2}f.
\end{equation}
\end{proposition}

\begin{proof}
Pick $f\in \tilde{\mathcal D}_0$. Since $\tilde{\mathcal D}_0$ is invariant under $A_1^k$, $k\ge 1$, in light of \eqref{e4} we obtain  from \eqref{e7}
\begin{equation}\label{e8}
\int_0^\infty t^{\alpha-1-k}[B_1e^{-tA_1}-B_2e^{-tA_2}]fdt=0,\quad k\ge 1.
\end{equation}
The change of variables $s=1/t$ gives
\[
\int_0^\infty s^{-\alpha-1+k}[B_1e^{-(1/s)A_1}-B_2e^{-(1/s)A_2}]fds=0,\quad k\ge 1.
\]
Substituting $k-1$ by $k$, we find
\[
\int_0^\infty s^{-\alpha+k}[B_1e^{-(1/s)A_1}-B_2e^{-(1/s)A_2}]fds=0,\quad k\ge 0.
\]
That is we have
\begin{equation}\label{e9}
\int_0^\infty s^kT(s)ds=0,\quad k\ge 0,
\end{equation}
where 
\[
T(s)=s^{-\alpha}[B_1e^{-(1/s)A_1}-B_2e^{-(1/s)A_2}]f,\quad s>0.
\]

Since the semigroups are bounded, we see   that $e^{-\lambda s}T\in L^1((0,1),D)$ for any $\lambda \in \mathbb{C}$. On the other hand, using  (WDG)  we infer
\begin{equation}\label{e10}
\|B_je^{-tA_j}f\|\le Ce^{-\mu/t },\quad j=1,2,\;  0<t\le 1,
\end{equation}
where $C=\max_{j=1,2}C_j(A_j,f,B_j)$ and $\mu=\min_{j=1,2}\mu_j(A_j,f,B_j)$.

Inequality \eqref{e10} yields
 \[
 \|e^{-\lambda s}T(s)\|\le 2Cs^{-\alpha}e^{-(\Re \lambda +\mu)s},\quad s\ge 1.
 \]
 Hence $e^{-\lambda s}T\in L^1((0,\infty),D)$ for any $\lambda\in \mathbb{C}$ with $\Re \lambda >-\mu$.
 
Let $L$ be the Laplace transform of $T$
\[
L(\lambda)=\int_0^\infty e^{-\lambda s}T(s)ds ,\quad \Re \lambda >-\mu.
\]
In that case \eqref{e9} means that
\[
L^{(k)}(0)=0,\quad k\ge 0.
\]
As $L$ is analytic in $\{z\in \mathbb{C};\; \Re z>-\mu\}$ we derive $L=0$. The uniqueness of the Laplace transform implies that  $T=0$. This  proves \eqref{e11}.
\end{proof}

\begin{corollary}\label{cor1}
Suppose that \eqref{e7} holds for all $f \in \tilde{\mathcal D}_0$. Then 
\begin{equation}\label{e001}
B_1e^{-tA_1}f=B_2e^{-tA_2}f, \quad f \in D_0.
\end{equation}
\end{corollary}

\begin{proof} Let $f \in D_0$. Then  $A_1f=A_2f \in \tilde{\mathcal D}_0$ and obtain from the previous proposition
\[
B_1e^{-tA_1}A_1f=B_2e^{-tA_2}A_2f.
\]
This means that
\[
\frac{d}{dt} B_1 e^{-tA_1}f = \frac{d}{dt} B_2 e^{-tA_2}f.
\]
We integrate this equality and take into account the assumption  $B_1f = B_2 f$ to obtain 
 the desired equality.
\end{proof}

\subsection{From the heat equation  to the wave equation: the transmutation formula}

Let $\mathcal{E}$ be a complex separable Hilbert space and $A$ a nonnegative self-adjoint operator on $\mathcal{E}$. Denote the spectral resolution of $A$ by $(E_\lambda)$. We recall that for every bounded Borel function $\varphi: [0,\infty )\rightarrow \mathbb{C}$, the operator $\varphi(A)$ is defined as follows 
\[
\varphi(A)=\int_0^\infty \varphi (\lambda )dE_\lambda.
\]
In particular, we have
\[
e^{-zA}=\int_0^\infty e^{-\lambda z}dE_\lambda,\quad \Re z>0,
\]
and
\[
\frac{\sin \left(tA^{1/2}\right)}{A^{1/2}}=\int_0^\infty \frac{\sin \left(t\lambda^{1/2}\right)}{\lambda^{1/2}}dE_\lambda,\quad t>0.
\]

Consider the wave equation
\begin{equation}\label{e12}
u''(t)+Au(t)=f(t),\quad u(0)=u'(0)=0.
\end{equation}

Let $f\in C_0^\infty((0,\infty),\mathcal{E})$. Then it is well known that  \eqref{e12} admits a unique solution $u=W_A(f)\in C^\infty ((0,\infty),\mathcal{E})$  given by
\begin{equation}\label{e13}
W_A(f)(t)=\int_0^t \frac{\sin \left((t-s)A^{1/2}\right)}{A^{1/2}}f(s)ds,\quad t\ge 0.
\end{equation}
This formula can be easily obtained from the results in \cite[Chapter 6, Section 2, page 490]{Ta} or even directly.

Let $\mathcal{E}_j$ be a complex separable Hilbert space and $A_j$ a nonnegative selfadjoint operator on $\mathcal{E}_j$, $j=1,2$. Let $D_0$ and $D$ be two non trivial closed subspaces of $\mathcal{E}_1\cap \mathcal{E}_2$ and let $B_j\in \mathscr{B}(\mathcal{E}_j,D)$, $j=1,2$.

\begin{proposition}\label{proposition2}
The condition
\begin{equation}\label{e15}
B_1e^{-tA_1}f=B_2e^{-tA_2}f,\quad  f\in D_0,
\end{equation}
yields
\begin{equation}\label{e14}
B_1W_{A_1}(f)=B_2W_{A_2}(f),\quad  f\in  C_0^\infty((0,\infty) ,D_0).
\end{equation}
\end{proposition}

\begin{proof}
 Similarly to \cite[(3.7)]{FGKU} we have the following transmutation formula
 \[
 e^{-tA_j}f=\frac{t^{-3/2}}{4\sqrt{\pi}}\int_0^\infty e^{-\tau/(4t)}\frac{\sin (\sqrt{\tau}A_j^{1/2})}{A_j^{1/2}}fd\tau,\quad f\in \mathcal{E}_j,\; t>0,\; j=1,2.
 \]
Under assumption \eqref{e15}, the preceding transmutation formula yields
\begin{align*}
&\int_0^\infty e^{-t\tau}B_1\frac{\sin (\sqrt{\tau}A_1^{1/2})}{A_1^{1/2}}fd\tau
\\
&\hskip3cm =\int_0^\infty e^{-t\tau}B_2\frac{\sin (\sqrt{\tau}A_2^{1/2})}{A_2^{1/2}}fd\tau,\quad f\in D_0,\; t>0.
\end{align*}
Fix $f\in D_0$. In that case we can rewrite the preceding formula as 
\[
\mathcal{L}_1(t)= \mathcal{L}_2(t),\quad t>0,
\]
where
\[
\mathcal{L}_j(\lambda)=\int_0^\infty e^{-\lambda \tau}B_j\frac{\sin (\sqrt{\tau}A_j^{1/2})}{A_j^{1/2}}fd\tau,\quad  \Re \lambda >0,\; j=1,2.
\]
Using the analyticity of $\mathcal{L}_j$, $j=1,2$ and the uniqueness of the Laplace transform, we obtain
\[
B_1\frac{\sin (\tau A_1^{1/2})}{A_1^{1/2}}f=B_2\frac{\sin (\tau A_2^{1/2})}{A_2^{1/2}}f,\quad \tau >0,\; f\in D_0.
\]
In light of  \eqref{e13}, we readily derive  \eqref{e14} from the last equality.
\end{proof}

\subsection{Proof of Theorem \ref{theorem0}}

Let $(M_1,g_1)$, $(M_2,g_2)$ satisfy the assumptions of Theorem \ref{theorem0} and set for simplicity of notation 
$A_1 = \Delta_{g_1}$ and $A_2 = \Delta_{g_2}$.  As we have seen in the preceding subsection, since $A_j$, $j=1,2$ is nonnegative selfadjoint operator, it is the generator of  an analytic semigroup $e^{-tA_{j}}$ with
\[
\| e^{-tA_{j}}\|\le 1,\quad t\ge 0.
\]
On the other hand, it is clear  that, for an arbitrary $\Omega \Subset O$, assumptions $\mathbf{(h1)}$ and $\mathbf{(h2)}$ hold with $D_0=C_0^\infty (\Omega)$ and $A_j= \Delta_{g_j}$, $j=1,2$.

The heat kernel associated to $A_{j}$ will be denoted in the sequel by $p_j(t,x,y)$. It is well known that $p_j(t,x,y)\in C^\infty ((0,\infty)\times M_j\times M_j)$, $j=1,2$ (e.g. \cite[Theorem 5.2.1]{Dav}).

Pick $x_0,y_0\in O$ with $x_0\ne y_0$. Then there exists $\Omega_j\Subset O$, $j=0,1$ satisfying $x_0\in \Omega_0$, $y_0\in \Omega_1$ and $\overline{\Omega}_0\cap \overline{\Omega}_1=\emptyset$.  Define $B_j\in \mathscr{B}(L^2(M_j,d\mu_{g_j}))$ as follows
\[
B_jh=\chi_{\Omega_0}h,\quad h\in L^2(M_j,d\mu_{g_j}).
\]

Let $f\in C_0^\infty (\Omega_1)$ and $h\in L^2(M_j,d\mu_{g_j})$.  It follows from Theorem \ref{theorem-appendix} in Subsection \ref{sub2} that  $e^{-tA_{j}}$ satisfies the Davies-Gaffney property. Thus
\begin{align*}
\left|\int_{M_j}\overline{h}B_je^{-tA_{j}}fd\mu_{g_j} \right|&=\left|\int_{\Omega_0}\overline{h}e^{-tA_{j}}fd\mu_{g_j}\right|
\\
&\qquad \le e^{-d^2/(4t)}\|h\|\|f\|,\quad t>0,\; j=1,2,
\end{align*}
where $d=\mathrm{dist}(\overline{\Omega}_0,\overline{\Omega}_1)$ and $\|\cdot\|$ denotes the norm of either $L^2(M_1,d\mu_{1})$ or $L^2(M_2,d\mu_{2})$.  Here
\[
\mathrm{dist}(\overline{\Omega}_0,\overline{\Omega}_1)=\inf\{d_g(x,y);\; x\in \overline{\Omega}_0,\; y\in \overline{\Omega}_1\},
\]
where $d_g$ is the Riemannian metric associated with $g:=g_1{_{|O_1}}=g_2{_{|O_2}}$.

Thus we have
\[
\|B_je^{-tA_{j}}f\|\le e^{-d^2/(4t)}\|f\|,\quad t>0,\; j=1,2.
\]
That is $(A_{j},B_j,f)$, $j=1,2$, satisfies the property (WDG).

We apply Corollary \ref{cor1} and obtain 
\[
B_1e^{-tA_{1}}f = B_2e^{-tA_{2}}f,\quad t>0,\; f\in C_0^\infty (\Omega_1).
\]
Thus, 
\[
\int_{\Omega_1}p_1(t,x_0,y)f(y)d\mu_g=\int_{\Omega_1}p_2(t,x_0,y)f(y)d\mu_g,\quad  t>0,\; f\in C_0^\infty (\Omega_1).
\]
Taking in this identify $f=f_k$, $k\ge 1$, where $(f_k)$ is a sequence of $C_0^\infty (\Omega_1)$ converging to $\delta_{y_0}$, we derive
\[
p_1(t,x_0,y_0)=p_2(t,x_0,y_0),\quad t>0.
\]
That is we have
\[
p_1(t,x_0,y_0)=p_2(t,x_0,y_0),\quad t>0,\; (x_0,y_0)\in \{(x,y)\in O\times O;\;x\ne y\}.
\]
A simple continuity argument shows that 
\begin{equation}\label{d1}
p_1(t,x_0,y_0)=p_2(t,x_0,y_0),\quad t>0,\; (x_0,y_0)\in O\times O.
\end{equation}
Let $D_0=C_0^\infty (O)$ and $B_j\in \mathscr{B}_j(L^2(M_j,d\mu_{g_j}),L^2(O))$ defined as follows
\[
B_jh=\chi_Oh,\quad h\in L^2(M_j,d\mu_{g_j}).
\]
Equality \eqref{d1} yields
\[
Be^{-tA_{1}}f=Be^{-tA_{2}}f,\quad f\in D_0
\]
which once combined with Proposition \ref{proposition2}, implies 
\[
\chi_OW_{A_{1}}f(t)=\chi_OW_{A_{2}}f(t),\quad t\ge 0,\; f\in C_0^\infty ((0,\infty)\times O).
\]
Theorem \ref{theorem0} follows then from \cite[Theorem 2]{HLOS}.

\subsection{Proof of Theorem \ref{theorem0.0}}

Let the assumptions of Theorem \ref{theorem0.0} be satisfied. From
\begin{align*}
\lambda_k^j\int_O\phi_k^j \overline{f}d\mu_{g_j}=\int_O-\Delta_{g_j}&\phi_k^j \overline{f}d\mu_{g_j}
\\
&=\int_O\phi_k^j \overline{\Delta_{g_j}f}d\mu_{g_j},\quad k\ge 0,\; j=1,2,\; f\in C_0^\infty (O),
\end{align*}
we derive that $A_1N_\ell=A_2N_\ell \subset N_\ell$.

We proceed as in the preceding proof in order to get
\begin{equation}\label{ee1}
B_1e^{-tA_{1}}h=B_2e^{-tA_{2}}h,\quad t>0,\; h\in N_\ell.
\end{equation}
Set
\[
T_\ell=\prod_{j=1,2,\; 0\le k\le \ell}(A_{g_j}-\lambda_k^j).
\]
Then one can check in a straightforward manner that $T_\ell(C_0^\infty(O))\subset N_\ell$. In particular \eqref{ee1} yields
\begin{equation}\label{ee2}
B_1e^{-tA_{1}}T_\ell f=B_2e^{-tA_{2}}T_\ell f,\quad t>0,\; f\in C_0^\infty(O).
\end{equation}
We multiply  this equality by  $e^{t \lambda_1^0}$ and set $T_\ell =(A_{1}-\lambda_1^0)\tilde{T}_\ell$. We can  rewrite \eqref{ee2} as follows
\begin{align}
B_1e^{-t(A_{1}-\lambda_1^0)}&(A_{1}-\lambda_1^0)\tilde{T}_\ell f \label{ee3}
\\
&=B_2e^{-t(A_{2}-\lambda_1^0)}(A_{2}-\lambda_1^0)\tilde{T}_\ell f,\quad t>0,\; f\in C_0^\infty(O),\nonumber
\end{align}
where we used the fact that the restriction to $C_0^\infty (O)$ of the operators $(A_{j}-\lambda_k^j)$, $j=1,2$, $0\le k\le n$, commute mutually, and $A_{1}$ and $A_{2}$ coincide on $C_0^\infty(O)$. That is we have
\begin{equation}\label{ee4}
\frac{d}{dt}B_1e^{-t(A_{1}-\lambda_1^0)}\tilde{T}_\ell f=\frac{d}{dt}B_2e^{-t(A_{2}-\lambda_1^0)}\tilde{T}_\ell f,\quad t>0,\; f\in C_0^\infty(O).
\end{equation}
As we have done in the proof of Theorem \ref{theorem0}, we find by integrating \eqref{ee4} 
\[
B_1e^{-t(A_{1}-\lambda_1^0)}\tilde{T}_\ell f=B_2e^{-t(A_{2}-\lambda_1^0)}\tilde{T}_\ell f,\quad t>0,\; f\in C_0^\infty(O).
\]
Thus,
\[
B_1e^{-tA_{1}}\tilde{T}_\ell f=B_2e^{-tA_{2}}\tilde{T}_\ell f,\quad t>0,\; f\in C_0^\infty(O).
\]
We repeat the same argument with $T_\ell$ substituted by $\tilde{T}_\ell$ to derive that \eqref{ee2} holds with $T_\ell^1$ instead of $T_\ell$ where
\[
T_\ell^m=\prod_{j=1,2,\; m\le k\le \ell}(A_j-\lambda_k^j),\quad 1\le m\le \ell.
\]
By an induction in $m$ we obtain in a final step
\[
B_1e^{-tA_{1}}f=B_2e^{-tA_{2}}f,\quad t>0,\; f\in C_0^\infty(O).
\]
The rest of the proof is exactly the same as that of Theorem \ref{theorem0}.

\section{Additional comments}

\subsection{From an initial-to-solution operator for the heat equation to an interior spectral data}

Let $(M_j,g_j)$, $j=1,2$, be a closed Riemannian manifold. Denote by $(\lambda_k^j)_{k\ge 1}$ the sequence of distinct eigenvalues of $\Delta_{g_j}$. The eigenspace associated with $\lambda_k^j$ is denoted by $E_k^j$, $k\ge 1$. We fix then $(\phi_{k,\ell}^j)_{1\le \ell \le m_k^j}$ an orthonormal basis of $E_k^j$, where $m_k^j=\mathrm{dim}(E_k^j)$, $k\ge 1$.

\begin{proposition}\label{proposition3}
Let $\ell \ge 0$ be an integer and $(M_i,g_j)$ be a closed Riemannian manifolds, $j=1,2$. Assume that there exists an open subset $O_j$ of $M_j$, $j=1,2$, such that $(O_1,g_1)=(O_2,g_2):=(O,g)$  and let $N_\ell$ be as in Theorem \ref{theorem0.0}. If 
\begin{equation}\label{sp1.0}
\chi_Oe^{-tA_{g_1}}f=\chi_Oe^{-tA_{g_2}}f,\quad t\ge 0,\;  f\in N_\ell,
\end{equation}
then, for all $k\ge 1$, $m_k^1=m_k^2:=m_k$ and
\[
\lambda_k^1=\lambda_k^2,\quad \phi_{k,\ell}^1{_{|O}}=P_k\phi_{k,\ell}^2{_{|O}},\quad 1\le \ell\le m_k,
\]
where $P_k$ is an orthogonal matrix.
\end{proposition}

\begin{proof}
First, let us see that from the proof of Theorem \ref{theorem0.0} we know that \eqref{sp1.0} implies
\begin{equation}\label{sp1}
\chi_Oe^{-tA_{g_1}}f=\chi_Oe^{-tA_{g_2}}f,\quad t\ge 0,\;  f\in C_0^\infty (O),
\end{equation}

Let $y\in O$ and $f\in C_0^\infty (O)$. We rewrite \eqref{sp1} in the form
\[
\sum_{k\ge 1}e^{-\lambda_k^1t}\sum_{\ell=1}^{m_k^1}(f,\phi_{k,\ell}^1)_g\phi_{k,\ell}^1(y)=\sum_{k\ge 1}e^{-\lambda_k^2t}\sum_{\ell=1}^{m_k^2}(f,\phi_{k,\ell}^2)_g\phi_{k,\ell}^2(y),\quad t\ge 0.
\]
By the uniqueness of Dirichlet series we derive that $\lambda_k^1=\lambda_k^2$, $k\ge 1$, and
\[
\sum_{\ell=1}^{m_k^1}(f,\phi_{k,\ell}^1)_g\phi_{k,\ell}^1(y)=\sum_{\ell=1}^{m_k^2}(f,\phi_{k,\ell}^2)_g\phi_{k,\ell}^2(y),\quad k\ge 1.
\]
Therefore
\[
\sum_{\ell=1}^{m_k^1}\phi_{k,\ell}^1(x)\phi_{k,\ell}^1(y)=\sum_{\ell=1}^{m_k^2}\phi_{k,\ell}^2(x)\phi_{k,\ell}^2(y),\quad k\ge 1,\; (x,y)\in O\times O.
\]
Since $(\phi_{k,1}^j,\ldots \phi_{k,m_k^j}^j)$ are linearly independent in $L^2(O)$ as a straightforward consequence of the unique continuation property of $\Delta_g$, we complete the proof by using \cite[Lemma 2.3]{CK}.
\end{proof}

In light of this proposition, we retrieve \cite[Proposition 2.1]{Fe} in which we removed the extra assumptions that $(O_j,g_j{_{|O_j}})$ belongs locally in a Gevery class $j=1,2$.

Proposition \ref{proposition3} says that, under the assumption \eqref{sp1}, we have $\lambda_k^1=\lambda_k^2$, $k\ge 1$, and there exist two orthonormal basis $(\varphi_k^1)$ and $(\varphi_k^2)$, consisting of eigenfunctions of $\Delta_{g_1}$ and $\Delta_{g_2}$ respectively, such that
\[
\varphi_k^1{_{|O}}=\varphi_k^2{_{|O}},\quad k\ge 0.
\]
We apply then \cite[Corollary 2]{HLOS} in order to find a diffeomorphism $F:M_1\rightarrow M_2$ such that $F^\ast g_2=g_1$.

In other words we proved the following theorem

\begin{theorem}\label{theoremh}
$\ell \ge 0$ be an integer and  $(M_j,g_j)$ be  closed Riemannian manifolds, $j=1,2$. Assume that there exists an open subset $O_j$ of $M_j$, $j=1,2$, such that $(O_1,g_1)=(O_2,g_2):=(O,g)$  and let $N_\ell$ be as in Theorem \ref{theorem0.0}. If 
\[
\chi_Oe^{-tA_{g_1}}f=\chi_Oe^{-tA_{g_2}}f,\quad t\ge 0,\;  f\in N_\ell,
\]
then we find a diffeomorphism $F:M_1\rightarrow M_2$ such that $F^\ast g_2=g_1$.
\end{theorem}

\subsection{Manifolds with boundary}

Let $(M,g)$ be a compact connected Riemannian manifold of dimension $n\ge 2$ with boundary $\partial M$. Let $(\lambda_k^g,\phi_k^g)$ be the sequence of eigenvalues and eigenfunctions of the operator $\Delta_g$ with domain $D(\Delta_g)=H_0^1(M)\cap H^2(M)$. Assume that $(\phi_k^g)$ form an orthonormal basis of $L^2(M,d\mu_g)$.

Let $\tilde{M}=M\times \{-1,1\}/\partial M$ be the double manifold made of two copies of $M$ where we identified the points of the boundary $(x,-1)$ and $(x,1)$, $x\in \partial M$. Such double manifold exists (e.g. \cite[Example 9.32, page 226]{Le}). However $g$ does not induce in general a metric on $\tilde{M}$. 

Denote by $R$ and $\nabla$ respectively the curvature and Levi-Cevita connection of $(M,g)$. Let $\zeta$ be a vector field defined in a neighborhood of $\partial M$ in $M$ whose integral curves are geodesics in $(M,g)$ emanating from $\partial M$ orthogonally.

Let $\tilde{g}=g$ in each copy of $M$. It is shown in \cite[Proposition]{Mo} that $(\tilde{M},\tilde{g})$ is a Riemannian manifold if and only if the following two conditions hold:
\\
(i) $\partial M$ is totally geodesic.
\\
(ii) $(\nabla_\zeta^{2k+1} R)(\zeta,\cdot )=0$ in $\partial M$, for each $k\ge 0$.
\\
When conditions (i) and (ii) are satisfied we call $(\tilde{M},\tilde{g})$ the Riemannian double of $(M,g)$.

We remark that, according to \cite[Corollary 1]{Mo}, if (i) holds and the Riemannian curvature of $(M,g)$ is parallel in a neighborhood of $\partial M$ then $(\tilde{M},\tilde{g})$ is the Riemannian double of $(M,g)$.
 
Define
\[
\dot{L}_2(M,d\mu_g)=\{f\in L^2(M,d\mu_g);\; \int_Mfd\mu_g=0\}.
\]

If $(M,g)$ admits a Riemannian double $(\tilde{M},\tilde{g})$ then the sequence of eigenvalues and an orthonormal basis of eigenfunctions of the operator 
\[
\dot{\Delta}_{\tilde{g}}:\dot{L}^2(\tilde{M},d\mu_{\tilde{g}})\rightarrow \dot{L}^2(\tilde{M},d\mu_{\tilde{g}}): u\mapsto \Delta_{\tilde{g}}u
\]
 with domain $D(\dot{\Delta}_{\tilde{g}})=\dot{L}^2(\tilde{M},d\mu_{\tilde{g}})\cap H^2(\tilde{M})$  are given by $\lambda_k^{\tilde{g}}=\lambda_k^g$ and
\[
\phi_k^{\tilde{g}}=\frac{1}{2}\left\{ \begin{array}{ll} -\phi_k^g\quad &\mathrm{in}\; M\times \{-1\}, \\ \phi_k^g &\mathrm{in}\; M\times \{1\}. \end{array}\right. 
\]
We refer to \cite[Theorem]{Mo} or \cite[Theorem 7]{BM} for a proof.

Denote by $i^\pm$ the isometric embedding
\[
i^\pm : M\rightarrow \tilde{M}: x\mapsto (x,\pm 1).
\]

Let $0<\alpha <1$ and for $O\Subset M$ set $O^+=O\times\{1\}$. Then we have 
\begin{align*}
\dot{\Delta}^{-\alpha}_{\tilde{g}}(i^+f)_{|O^+}&=\sum_{k\ge 1}(\lambda_k^{\tilde{g}})^{-\alpha}(i^+f|\phi_k^{\tilde{g}})_{\tilde{g}}\phi_k^{\tilde{g}}{_{|O^+}}
\\
&=\sum_{k\ge 1}(\lambda_k^g)^{-\alpha}(f|\phi_k^g)_g\phi_k^g{_{|O}},\quad f\in C_0^\infty (O).
\end{align*}
That is the following identity holds
\[
\dot{\Delta}^{-\alpha}_{\tilde{g}}(i^+f)_{|O^+}=\Delta^{-\alpha}_gf{_{|O}},\quad f\in C_0^\infty (O).
\]

We readily obtain from Theorem \ref{theorem0.0} that $(O,f\in N_\ell \mapsto \Delta^{-\alpha}_gf{_{|O}})$ determines uniquely $g$ up to isometry, where $\ell \ge 0$ is arbitrarily fixed.

\appendix

\section{}\label{appendixA}

\subsection{Analytic semigroups} \label{sub1}

For reader convenience, we provide in this short subsection the definition of analytic semigroups. We refer the reader to the textbook \cite{RR} for an introductory study of semigroups (see also the classical monographs on the subject \cite{Ka,Pa}).

Let $E$ be a Banach space. A family $(T(t))_{t\ge 0}$ of bounded operators is called a strongly continuous semigroup if it satisfies the following properties: (i) $T(t+s)=T(t)T(s)$, $t,s\ge 0$ ; (ii) $T(0)=I$ ; (iii) for any $u\in E$, $t\in [0,\infty)\mapsto T(t)u\in E$ is continuous.

The infinitesimal generator of a strongly continuous semigroup $(T(t))_{t\ge 0}$ is defined by
\[
Au:=\lim_{t\downarrow 0}\frac{T(t)u-u}{t},
\]
and the domain of $A$, usually denoted $D(A)$, is the set of all vectors $u\in E$ for which this limit exists.

Traditionally, the semigoup whose infinitesimal generator is $A$ is denoted by $e^{tA}$.

A strongly continuous semigroup $e^{tA}$ is called analytic (or holomorphic) semigroup if: (i') there exists $\theta \in (0,\pi/2)$ such that conditions (i), (ii) and (iii) hold with $t,s\in \Sigma_\theta=\{0\}\cup\{z\in \mathbb{C};\; |\arg z|<\theta\}$ ; (ii') $t\in \Sigma_\theta \setminus\{0\} \mapsto e^{tA}$ is analytic in the sense of the uniform operator topology.

\subsection{Davies-Gaffney estimate}\label{sub2}

Let $(M,d,\mu)$ be a measure space. Here $\mu$ is a Borel measure with respect to the topology defined by the metric $d$. Let $\mathbb{C}_+=\{z\in \mathbb{C};\; \Re z>0\}$. Let $F:\mathbb{C}_+\mapsto \mathscr{B}(L^2(M,d\mu))$ satisfying
\[
\|F(z)\|_{\mathscr{B}(L^2(M,d\mu))} \le 1,\quad z\in \mathbb{C}_+,
\]
and $z\in \mathbb{C}_+\mapsto (F(z)f|g)$ is analytic for each $f,g\in L^2(M,d\mu)$, $(\cdot|\cdot)$ is the usual scalar product of $L^2(M,d\mu)$. Let $\|\cdot \|_2$ denotes the norm of $L^2(M,d\mu)$. We say that $F$ (or the family $(F(z))_{z\in \mathbb{C}_+}$) admits the Davies-Gaffney estimate if 
\[
|(F(t)f|g)|\le e^{-d^2/4t}\|f\|_2\|g\|_2,
\]
for all $t>0$, $f,g\in L^2(M,d\mu)$ satisfying $\mathrm{supp}(f)\subset O_1$, $\mathrm{supp}(g)\subset O_2$, where $O_1\subset M$, $O_2\subset M$ are such that
\[
d=\mathrm{dist}(O_1,O_2)=\inf\{d(x,y);\; x\in O_1,\; y\in O_2\}.
\]
This definition is borrowed from \cite{CS} where the authors proved that the definition of Davies-Gaffney estimate above is equivalent to the following one:
\[
|(F(t)\chi_{O_1}|\chi_{O_2})|\le e^{-d^2/4t}\mu(O_1)^{1/2}\mu(O_1)^{1/2},
\]
for all $t>0$ and Borel sets $O_1$, $O_2$ satisfying $\mu (O_1)<\infty$ and $\mu (O_2)<\infty$.

In light of this observation we have from \cite[Theorem 2, page 103]{Da} (see also \cite[Theorem 3.2, page 154]{Gr}) the following theorem

\begin{theorem}\label{theorem-appendix}
Let $(M,g)$ be a complete Riemannian manifold. Then $e^{-t\Delta_g}$ admits the Davies-Gaffeny estimate.
\end{theorem}

\end{document}